\documentclass[a4paper,abstracton]{scrartcl}

\usepackage[utf8]{inputenc}

\usepackage{amsmath}
\usepackage{amsthm}
\usepackage{amssymb}

\usepackage[usenames,dvipsnames]{color}
\usepackage{graphicx}
\usepackage{subfigure}

\usepackage{amssymb,amsmath,verbatim,tabularx,enumitem,colonequals,graphicx,psfrag,subfigure}

\usepackage{cite}

\newcommand{\cU}{\mathcal{U}}

\newcommand{\X}{\mathcal{X}}

\usepackage{authblk}
\usepackage[normalem]{ulem} 

\usepackage{framed}
\usepackage{rotating}

\newtheorem{theorem}{Theorem}
\newtheorem{lemma}[theorem]{Lemma}

\newtheorem{definition}[theorem]{Definition}

\newtheorem{example}[theorem]{Example}

\newcommand{\BIGOP}[1]{\mathop{\mathchoice%
{\raise-0.22em\hbox{\huge $#1$}}%
{\raise-0.05em\hbox{\Large $#1$}}{\hbox{\large $#1$}}{#1}}}

\allowdisplaybreaks

\begin{document}

\title{Constructing Representative Scenarios to Approximate Robust Combinatorial Optimization Problems}

\author{Marc Goerigk\thanks{Email: m.goerigk@lancaster.ac.uk}}
\affil{Department of Management Science, Lancaster University, United Kingdom}

\date{}

\maketitle

\abstract{In robust combinatorial optimization with discrete uncertainty, two general approximation algorithms are frequently used, which are both based on constructing a single scenario representing the whole uncertainty set. In the midpoint method, one optimizes for the average case scenario. In the element-wise worst-case approach, one constructs a scenario by taking the worst case in each component over all scenarios. Both methods are known to be $N$-approximations, where $N$ is the number of scenarios.

In this paper, these results are refined by reconsidering their respective proofs as optimization problems. We present a linear program to construct a representative scenario for the uncertainty set, which guarantees an approximation guarantee that is at least as good as for the previous methods. Incidentally, we show that the element-wise worst-case approach can have an advantage over the midpoint approach if the number of scenarios is large. In numerical experiments on the selection problem we demonstrate that our approach can improve the approximation guarantee of the midpoint approach by around 20\%.}

\textbf{Keywords: } robust optimization; combinatorial optimization; approximation algorithms

\section{Introduction}

We consider combinatorial optimization problems of the general form
\[ \min_{\pmb{x}\in\X} \pmb{c}\pmb{x} \]
where $\pmb{c} \ge \pmb{0}$ is a cost vector, and $\X \subseteq \{0,1\}^n$ is a set of feasible solutions. As real-world problems may suffer from uncertainty, robust counterparts to combinatorial problems have been considered in the literature, see \cite{Aissi2009,kasperski2016robust} for surveys on the topic. The resulting robust (or min-max) optimization problem is then of the form
\[ \min_{\pmb{x}\in\X} \max_{\pmb{c}\in\cU} \pmb{c}\pmb{x} \tag{\textsc{MinMax}}\]
where $\cU$ contains all possible cost vectors $\pmb{c}^1, \ldots, \pmb{c}^N$ against we wish to protect.

As robust combinatorial problems are usually NP-hard, approximation methods have been considered \cite{Aissi2007281}. Two such heuristics stand out in the literature, as they are easy to use and implement, and have been providing the best-known approximation guarantee for a wide range of problems. While this guarantee has been improved for specific problems, they are still the best-known general methods (see \cite{approx}).
Both algorithms are based on constructing a single scenario that represents the whole uncertainty $\cU$. For the midpoint algorithm, we use $\hat{\pmb{c}}$ with $\hat{c}_i = 1/N \sum_{j\in[N]} c^j_i$ for all $i\in[n]$. For the element-wise worst-case algorithm, we set $\overline{\pmb{c}}$ by using $\overline{c}_i = \max_{j\in[N]} c^j_i$. Let us denote by $\pmb{x}(\pmb{c})$ a minimizer for the nominal problem with costs $\pmb{c}$, and set $\hat{\pmb{x}}:=\pmb{x}(\hat{\pmb{c}})$ (the midpoint solution) and $\overline{\pmb{x}}:=\pmb{x}(\overline{\pmb{c}})$ (the element-wise worst-case solution). The following results can be found in \cite{Aissi2009}.

\begin{theorem}\label{th-n1}
The midpoint solution $\hat{\pmb{x}}$ is an $N$-approximation for \textsc{MinMax}.
\end{theorem}

\begin{theorem}\label{th-n2}
The element-wise worst-case solution $\overline{\pmb{x}}$ is an $N$-approximation for \textsc{MinMax}.
\end{theorem}

Frequently, problems with ''nice'' structure (such as shortest path, spanning tree, selection, or assignment) have been considered in the literature, where it is possible to solve the nominal problem in polynomial time. In particular, this setting makes it possible to solve both of the above approaches in polynomial time by solving one specific scenario (i.e., finding $\pmb{x}(\hat{\pmb{c}})$ or $\pmb{x}(\overline{\pmb{c}})$). This can then be used, e.g., as part of a branch and bound procedure for the (hard) robust problem.

Recently, data-driven robust optimization approaches have been investigated in the literature (see, e.g., \cite{atmos,bertsimas2018data}). This paper has a similar research outlook by using the available data for better approximation guarantees, instead of ignoring structure that may be present. In a similar spirit, by analyzing the symmetry of an uncertainty set, \cite{conde2012constant} is able to derive improved approximation bounds for the related \textsc{MinMax Regret} problem with compact uncertainty sets.

The contributions of this paper are as follows. By re-examining the proofs for Theorems~\ref{th-n1} and \ref{th-n2}, we present a linear program (LP) to construct a scenario $\pmb{c}'$ that is ''representative'' for the uncertainty set $\cU$. We show that the resulting solution $\pmb{x}(\pmb{c}')$ has an approximation guarantee that is at least as good as the guarantee for $\hat{\pmb{x}}$ and $\overline{\pmb{x}}$. We also compare the midpoint and element-wise worst-case approach in more detail and find that the latter can outperform the former if the number of scenarios is large. In numerical experiments, we compare the quality of upper and lower bounds of our approach with the midpoint method, and demonstrate that it is possible to find considerably smaller a-priori and a-posteriori gaps by solving a simple linear program.

\section{Scenario construction based on the midpoint approach}\label{mainsec}

Let $OPT$ be the optimal objective value of problem \textsc{MinMax}, and let $\pmb{x}^*$ be any optimal solution. We make the following distinctions.

\begin{definition}
Let some scenario $\pmb{c}$ (not necessarily in $\cU$) be given. Then
\[ UB(\pmb{c}) = \max_{i\in[N]} \pmb{c}^i\pmb{x}(\pmb{c}) \]
is an upper bound on $OPT$.
If it is possible to compute a lower bound from $\pmb{c}$, we denote this as $LB(\pmb{c})$, and a bound on the ratio as
\[ r(\pmb{c}) \ge UB(\pmb{c}) / LB(\pmb{c}) \]
We call $r(\pmb{c})$ an \emph{a-priori} bound, if it does not require the computation of $\pmb{x}(\pmb{c})$ to find. Otherwise, we call it an \emph{a-posteriori} bound.
\end{definition}
The reason for this distinction is that calculation of $\pmb{x}$ can be costly, if the nominal problem is not solvable in polynomial time.

As an example, the midpoint method uses $\hat{\pmb{c}} := \frac{1}{N}\sum_{i\in[N]} \pmb{c}^i$. It comes with an a-priori bound that is $N$, but by using $LB(\hat{\pmb{c}}) = \hat{\pmb{c}}\pmb{x}(\hat{\pmb{c}})$, we can calculate a stronger a-posteriori bound.

We now consider the problem of finding a better a-priori bound than $N$. To this end, note that Theorem~\ref{th-n1}  can be proven in the following way.
\begin{proof}[Proof of Theorem~\ref{th-n1}]
\[
UB(\hat{\pmb{c}}) = \max_{i\in[N]} \pmb{c}^i \hat{\pmb{x}} \stackrel{(i)}{\le} N \hat{\pmb{c}} \hat{\pmb{x}} \le N \hat{\pmb{c}} \pmb{x}^* \stackrel{(ii)}{\le} N \max_{i\in[N]} \pmb{c}^i \pmb{x}^* = N\cdot OPT
\]
\end{proof}
To mirror the steps of this proof, let us consider the following optimization problem:
\begin{align}
\min_{t,\pmb{c}}\ &t \label{eq0} \\
\text{s.t. } & \max_{i\in[N]} \pmb{c}^i\pmb{x}(\pmb{c}) \le t\cdot \pmb{c}\pmb{x}(\pmb{c}) \label{eq1}\\
& \pmb{c} \pmb{x}^* \le \max_{i\in[N]} \pmb{c}^i \pmb{x}^*  \label{eq2}
\end{align}
\begin{lemma}\label{lem1}
Let $(t,\pmb{c})$ be a feasible solution to problem (\ref{eq0}--\ref{eq2}). Then, $\pmb{x}(\pmb{c})$ is a $t$-approximation for \textsc{MinMax}.
\end{lemma}
\begin{proof}
Analogous to the proof of Theorem~\ref{th-n1}.
\end{proof}
Note that Problem~(\ref{eq0}--\ref{eq2}) cannot be solved directly, as both the optimal solution $\pmb{x}^*$ and $\pmb{x}(\pmb{c})$ are unknown. To circumvent these two issues, we use different, sufficient constraints instead.

\begin{lemma}\label{lem2}
Let $\pmb{c}$ fulfil
\begin{equation}
\sum_{j\in S} c^i_j \le t \sum_{j\in S} c_j \quad \forall i\in[N], S\subseteq[n] : |S| = k \label{suf1}
\end{equation}
for some value of $t$, and constant $k$ such that $k\le \sum_{j\in[n]} x_j$ for all $x\in\X$. Then, $(t,\pmb{c})$ also fulfils \eqref{eq1}.
\end{lemma}
\begin{proof}
Let $X = \{j\in[n]: x_j(\pmb{c}) = 1\}$ and $\mathcal{S} = \{S\subseteq[n] : |S| = k, S\subseteq X\}$. Then, the number of sets $S$ in $\mathcal{S}$ containing a specific item $j\in X$ is the same for all $j$. Let $\ell$ be this number. By summing \eqref{suf1} over all $S\in\mathcal{S}$, we find that
\[ \ell \sum_{j\in X} c^i_j \le t \ell \sum_{j\in X} c_j \qquad \forall i\in[N] \]
and the claim follows.
\end{proof}
Note that for constant $k$, it is possible in polynomial time to check if $k\le \sum_{j\in[n]} x_j$ for all $x\in\X$. Also, the set $\mathcal{S}$ contains polynomially many elements. As an example, for $k=1$, Constraint~\eqref{suf1} becomes
\[ c^i_j \le tc_j \qquad \forall i\in[N], j\in[n] \]
and for $k=2$, it becomes
\[ c^i_j + c^i_l \le t(c_j+c_l) \qquad \forall i\in[N], j,l\in[n], j\neq l \]
In general, the constraints for some fixed $k$ also imply the constraints for any larger $k$. This means that the larger the value of $k$, the larger is the set of feasible solutions to our optimization problem, and the better approximation guarantees we can get.

\begin{lemma}\label{lem3}
Let $\pmb{c}$ be in $conv(\cU) = conv\{\pmb{c}^1,\ldots,\pmb{c}^N\}$. Then, $\pmb{c}$ fulfils \eqref{eq2}.
\end{lemma}
\begin{proof}
Let $\pmb{c} = \sum_{i\in[N]} \lambda_i \pmb{c}^i$ with $\sum_{i\in[N]} \lambda_i = 1$ and $\lambda_i \ge 0$ for all $i\in[N]$. Then, for any $\pmb{x}\in\X$,
\[ \pmb{c}\pmb{x} = \sum_{i\in[N]} \lambda_i \pmb{c}^i\pmb{x} \le \sum_{i\in[N]} \lambda_i \max_{j\in[N]}\pmb{c}^j\pmb{x} = \max_{i\in[N]} \pmb{c}^i\pmb{x} \]
\end{proof}
We now consider the following linear program:
\begin{align}
\max\ & t \label{neq0}\\
\text{s.t. } & t \sum_{j\in S} c^i_j \le \sum_{j\in S} c_j & \forall i\in[N], S\subseteq[n] : |S| = k \label{neq1} \\
& \pmb{c} = \sum_{i\in[N]} \lambda_i \pmb{c}^i \label{neq2}\\
& \sum_{i\in[N]} \lambda_i = 1 \label{neq3}\\
& \lambda_i \ge 0 & \forall i\in[N] \label{neq4}
\end{align}
Note that we replaced variable $t$ in Problem~(\ref{eq0}--\ref{eq2}) with $1/t$ to linearize terms.

\begin{theorem}
Let $(t^*,\pmb{c}^*)$ be an optimal solution to Problem~(\ref{neq0}--\ref{neq4}). Then, $\pmb{x}(\pmb{c}^*)$ is a $1/t^*$-approximation for \textsc{MinMax}, and $1/t^* \le N$.
\end{theorem}
\begin{proof}
By Lemmas~\ref{lem2} and \ref{lem3}, $(1/t^*,\pmb{c}^*)$ is feasible for Problem~(\ref{eq0}--\ref{eq2}). Using Lemma~\ref{lem1}, we therefore find that $\pmb{x}(\pmb{c}^*)$ is a $1/t^*$-approximation for \textsc{MinMax}.

To see that $1/t^* \le N$, note that $(1/N,\hat{\pmb{c}})$ is a feasible solution to Problem~(\ref{neq0}--\ref{neq4}).
\end{proof}

Once a solution $(t^*,\pmb{c}^*)$ has been computed, we have found an a-priori approximation guarantee. If we then compute $\pmb{x}(\pmb{c}^*)$, we can derive a lower bound $\pmb{c}^*\pmb{x}(\pmb{c}^*)$, as $\pmb{c}^*\in conv(\cU)$, and an upper bound by calculating the objective value of $\pmb{x}(\pmb{c}^*)$ for \textsc{MinMax}. This way, a stronger a-posteriori guarantee is found.

\begin{example}
We illustrate our approach using a small selection problem as an example. Given four items, the task is to choose two of them that minimize the worst-case costs over three scenarios. The upper part of Table~\ref{extable} shows the item costs in each scenario.

\begin{table}[htb]
\begin{center}
\begin{tabular}{c|rrrr}
item & 1 & 2 & 3 & 4\\
\hline
$\pmb{c}^1$ & 5 & 5 & 3 & 3 \\
$\pmb{c}^2$ & 3 & 8 & 9 & 7 \\
$\pmb{c}^3$ & 3 & 2 & 1 & 6 \\
\hline
$\hat{\pmb{c}}$ & 3.67 & 5.00 & 4.33 & 5.33 \\
$\pmb{c}'$ & 3.75 & 6.88 & 6.75 & 5.50 \\
$\pmb{c}''$ & 3.00 & 8.00 & 9.00 & 7.00
\end{tabular}
\caption{Example item costs, with midpoint scenario ($\hat{\pmb{c}}$), our LP-based scenario with $k=1$ ($\pmb{c}'$), and with $k=2$ ($\pmb{c}''$).}\label{extable}
\end{center}
\end{table}

The midpoint scenario (i.e., the average in each item) is shown in the row below ($\hat{\pmb{c}}$). An optimal solution for this scenario is to pack items 1 and 3. This means that we have an a-priori approximation ratio of $N=3$, and can calcluate a lower bound $LB(\hat{\pmb{c}}) = \hat{\pmb{c}}\hat{\pmb{x}} = 8$ and an upper bound $UB(\hat{\pmb{c}}) = \max_{i\in[N]}\pmb{c}^i\hat{\pmb{x}} = 12.$ Combining lower and upper bound, we find the stronger a-posteriori bound of $1.50$.

Using our linear program~(\ref{neq0}--\ref{neq4}) with $k=1$, we construct the scenario given in the next row $(\pmb{c}')$ and find an a-priori guarantee of $1.33$. For this scenario, an optimal solution is to take items 1 and 4. Accordingly, we find a lower bound of $9.25$, an upper bound of $10$, and an a-posteriori ratio of $1.08$.

Finally, we also use our LP with $k=2$ to find the scenario $\pmb{c}''$ and an a-priori guarantee of 1. This means that even before we have solved the problem, we already know that the resulting solution will be optimal. Indeed, we find that packing items 1 and 4 gives the optimal solution with objective value 10.
\end{example}

Note that we can also use the linear program~(\ref{neq0}--\ref{neq4}) to strengthen the approximation guarantee of the midpoint scenario $\hat{\pmb{c}}$ without calculating $\hat{\pmb{x}}$, by only keeping $t$ variable.

We conclude this section by introducing an alternative approach to calculate a-posteriori bounds, which cannot be used for a-priori bounds. To this end, note that
\[ \max_{\pmb{c}\in conv(\cU)} \min_{\pmb{x}\in\X} \pmb{c}\pmb{x} \le  \min_{\pmb{x}\in\X} \max_{i\in[n]}  \pmb{c}^i\pmb{x}\]
If the nominal problem can be written as a linear program, it can be dualized to find a compact formulation for the max-min problem. As both $\hat{\pmb{c}}$ and the optimal solution to problem (\ref{neq0}--\ref{neq4}) are in $conv(\cU)$, this approach will result in a lower bound which will be at least as good as the lower bounds of the other two approaches. This may not result in a better ratio beteen upper and lower bound, however. We will test this approach in the experimental section.

\section{On the element-wise worst-case}

We now focus on the element-wise worst-case scenario $\overline{\pmb{c}}$ with $\overline{c}_i = \max_{j\in[N]} c^j_i$. A proof for Theorem~\ref{th-n2} is the following.
\begin{proof}[Proof of Theorem~\ref{th-n2}]
\[ UB(\overline{\pmb{c}}) = \max_{i\in[N]} \pmb{c}^i \overline{\pmb{x}} \stackrel{(i)}{\le} \overline{\pmb{c}}\overline{\pmb{x}} \le \overline{\pmb{c}}\pmb{x}^* \stackrel{(ii)}{\le} N \max_{i\in[N]} \pmb{c}^i\pmb{x}^* = N \cdot OPT\]
\end{proof}
Accordingly, we can generalize this proof to an optimization problem by writing
\begin{align}
\min_{t,\pmb{c}}\ &t \label{wc0}\\
\text{s.t. } & \max_{i\in[N]} \pmb{c}^i \pmb{x}(\pmb{c}) \le \pmb{c}\pmb{x}(\pmb{c}) \label{wc1}\\
& \pmb{c}\pmb{x}^* \le t \max_{i\in[N]} \pmb{c}^i \pmb{x}^* \label{wc2}
\end{align}
By substituting $\pmb{c}' := \pmb{c}/t$, Problem~(\ref{wc0}--\ref{wc2}) becomes equivalent to Problem~(\ref{eq0}--\ref{eq2}). Hence, we can apply the same techniques to transform this into a conservative linear program~(\ref{neq0}--\ref{neq4}) as in the previous section. Note, however, that while $\hat{\pmb{c}}$ is a feasible solution for this problem, this may not be the case for $\overline{\pmb{c}}$.

Related to the \textsc{MinMax} approach is \textsc{MinMax Regret}, where objective values are normalized by the optimal objective value in each scenario, i.e., 
\[ \min_{\pmb{x}\in\X} \max_{i\in[N]} \left( \pmb{c}^i\pmb{x} - \pmb{c}^i\pmb{x}(\pmb{c}^i) \right)  \tag{\textsc{MinMax Regret}} \]
The following result is also from \cite{Aissi2009}.

\begin{theorem}\label{th-n3}
The midpoint algorithm is an $N$-approximation for \textsc{MinMax Regret}; this does not hold for the element-wise worst-case algorithm.
\end{theorem}

In combination with Theorems~\ref{th-n1} and \ref{th-n2}, this means that there are no known problem classes where the element-wise worst-case solution gives a better performance guarantee than the midpoint solution. The midpoint solution has also been found to be the best-known general approximation algorithm for interval uncertainty problems \cite{kasperski2006approximation}. For these reasons, the midpoint solution has seen more attention in the research literature than the element-wise worst-case approach. However, in the following we show that if the number of scenarios is large, we element-wise worst-case approach can perform better than the midpoint approach, i.e., not only the size of the uncertainty set plays a role for approximability, but also the problem dimension.

\begin{theorem}\label{maintheorem}
The element-wise worst-case algorithm is a $|X|$-approximation for \textsc{MinMax}, where 
$|X|=\max_{\pmb{x}\in\X} \sum_{j\in[n]} x_j$. 
\end{theorem}
\begin{proof}
It holds that
\begin{align*}
&\max_{i\in[N]} \sum_{j\in[n]} c^i_j \overline{x}_j 
\le \sum_{j\in[n]} \overline{c}_j \overline{x}_j 
\le \sum_{j\in[n]} \overline{c}_j x^*_j 
= \sum_{j\in[n]} \max_{i\in[N]} c^i_j x^*_j \\
& \le |X| \cdot \max_{j\in[n]} \max_{i\in[N]} c^i_j x^*_j
= |X|\cdot \max_{i\in[N]} \max_{j\in[n]} c^i_j x^*_j 
\le |X| \cdot \max_{i\in[N]} \sum_{j\in[n]} c^i_j x^*_j = |X| \cdot OPT
\end{align*}
\end{proof}

Note that $|X| \le n$. The approximation guarantees from Theorems~\ref{th-n1} and \ref{th-n2} are tight, as the following two examples for robust shortest path problems demonstrate (see also \cite{Aissi2009}).

\begin{figure}[htbp]
\centering
\subfigure[Hard instance for the midpoint solution.\label{graph1}]{\includegraphics[width=.25\textwidth]{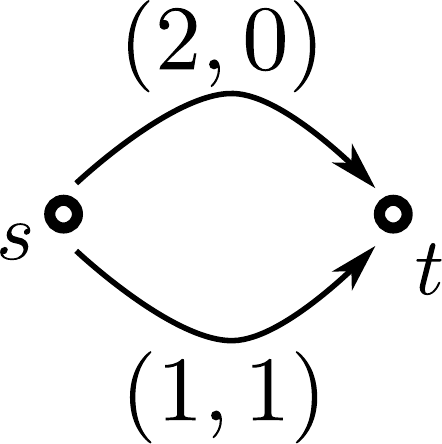}}\hspace*{2cm}
\subfigure[Hard instance for the element-wise worst-case solution.\label{graph2}]{\includegraphics[width=.35\textwidth]{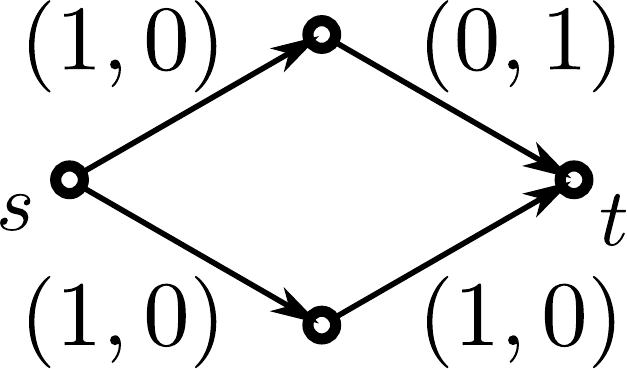}}
\caption{Example instances for robust shortest path with two scenarios.}\label{graphs}
\end{figure}

In Figure~\ref{graph1}, the midpoint solution cannot distinguish between the upper edge and the lower edge. Hence, in this case, the $N$-approximation guarantee is tight with $N=2$. In Figure~\ref{graph2}, the element-wise worst-case solution cannot differentiate between the upper and the lower path. This instance is an example where the $N$-approximation guarantee is tight for this approach.

Note that the instance from Figure~\ref{graph1} can be extended by using more scenarios, preserving that the midpoint solution is an $N$-approximation, without additional edges. This is not the case for the element-wise worst-case scenario in Figure~\ref{graph2}: To extend this instance to more scenarios, additional edges are required. This demonstrates that the midpoint solution is not a $|X|$-approximation, as shown for the element-wise worst-case approach.

\section{Experiments}

To test the quality of our LP-based scenario construction approach, we consider instances of the selection problem (see, e.g., \cite{kasperski2016robust}). Here, $\X = \{ \pmb{x}\in\{0,1\}^n : \sum_{j\in[n]} x_j = p\}$ for some integer parameter $p$. We generate item costs $c^i_j$ by sampling uniformly i.i.d. from $\{0,1,\ldots,100\}$. We use instances sizing from $n=10$, $p=3$ to $n=30$, $p=9$ and use $N\in\{2,5,10,50,100\}$. For each parameter combination, we generate 1000 instances and average results.

Table~\ref{apriori} shows the a-priori bounds for the midpoint approach when using our linear program~(\ref{neq0}--\ref{neq4}) for evaluation with $k=1$, $k=2$ and $k=3$ (Mid-1-Pre, Mid-2-Pre, and Mid-3-Pre, respectively). We compare this to the a-priori bounds that are found when also optimizing over the scenario $\pmb{c}$ for $k=1$, $k=2$ and $k=3$ (LP-1-Pre, LP-2-Pre, and LP-3-Pre, respectively). Note that overall, all guarantees are considerably smaller than $N$. Furthermore, our approach is able to improve the bound of the midpoint algorithm. On average, the guarantee that the midpoint approach gives is more than 20\% larger than our guarantee.

\begin{table}
\centering
\begin{tabular}{rrr|rrr|rrr}
$n$ & $p$ & $N$ & Mid-1-Pre & Mid-2-Pre & Mid-3-Pre & LP-1-Pre & LP-2-Pre & LP-3-Pre \\
\hline
10 & 3 & 2 & 1.86 & 1.75 & 1.65 & 1.70 & 1.57 & 1.46 \\
10 & 3 & 5 & 2.41 & 2.09 & 1.90 & 1.83 & 1.67 & 1.54 \\
10 & 3 & 10 & 2.45 & 2.13 & 1.97 & 1.79 & 1.65 & 1.53 \\
10 & 3 & 50 & 2.26 & 2.10 & 2.00 & 1.59 & 1.53 & 1.46 \\
10 & 3 & 100 & 2.18 & 2.08 & 2.00 & 1.52 & 1.48 & 1.43 \\
20 & 6 & 2 & 1.93 & 1.86 & 1.80 & 1.84 & 1.76 & 1.70 \\
20 & 6 & 5 & 2.66 & 2.32 & 2.14 & 2.09 & 1.94 & 1.82 \\
20 & 6 & 10 & 2.63 & 2.32 & 2.16 & 2.01 & 1.89 & 1.80 \\
20 & 6 & 50 & 2.32 & 2.18 & 2.09 & 1.77 & 1.73 & 1.69 \\
20 & 6 & 100 & 2.23 & 2.13 & 2.06 & 1.70 & 1.67 & 1.64 \\
30 & 9 & 2 & 1.96 & 1.92 & 1.87 & 1.90 & 1.84 & 1.79 \\
30 & 9 & 5 & 2.78 & 2.45 & 2.27 & 2.24 & 2.08 & 1.97 \\
30 & 9 & 10 & 2.73 & 2.42 & 2.26 & 2.13 & 2.03 & 1.94 \\
30 & 9 & 50 & 2.36 & 2.22 & 2.14 & 1.87 & 1.83 & 1.79 \\
30 & 9 & 100 & 2.26 & 2.16 & 2.10 & 1.79 & 1.77 & 1.74
\end{tabular}
\caption{Average a-priori bounds.}\label{apriori}
\end{table}

We contrast the a-priori bounds with a-posteriori bounds in Table~\ref{aposteriori}, i.e., we calculate the solutions $\pmb{x}(\pmb{c})$ for the respective scenarios $\pmb{c}$ and the resulting ratio of upper and lower bound. On average, the bound provided by the midpoint solution is around $17\%$ larger than the bound provided by our approach with $k=2$ or $k=3$. The max-min approach (denoted by MM) performs slightly better than our approach (Mid-Post is on average $19\%$ larger than MM-Post), but this comes without an a-priori guarantee, at the cost of higher computational effort, and it is not always possible to compute as explained in Section~\ref{mainsec}.

\begin{table}
\centering
\begin{tabular}{rrr|r|rrr|r}
$n$ & $p$ & $N$ & Mid-Post & LP-1-Post & LP-2-Post & LP-3-Post & MM-Post \\
\hline
10 & 3 & 2 & 1.30 & 1.24 & 1.22 & 1.21 & 1.24 \\
10 & 3 & 5 & 1.57 & 1.35 & 1.30 & 1.32 & 1.29 \\
10 & 3 & 10 & 1.66 & 1.39 & 1.34 & 1.36 & 1.34 \\
10 & 3 & 50 & 1.82 & 1.37 & 1.36 & 1.38 & 1.37 \\
10 & 3 & 100 & 1.85 & 1.35 & 1.35 & 1.36 & 1.35 \\
20 & 6 & 2 & 1.21 & 1.18 & 1.17 & 1.16 & 1.14 \\
20 & 6 & 5 & 1.40 & 1.30 & 1.26 & 1.24 & 1.19 \\
20 & 6 & 10 & 1.47 & 1.33 & 1.28 & 1.28 & 1.24 \\
20 & 6 & 50 & 1.59 & 1.34 & 1.31 & 1.32 & 1.32 \\
20 & 6 & 100 & 1.63 & 1.33 & 1.31 & 1.32 & 1.32 \\
30 & 9 & 2 & 1.17 & 1.16 & 1.15 & 1.14 & 1.10 \\
30 & 9 & 5 & 1.32 & 1.26 & 1.21 & 1.20 & 1.14 \\
30 & 9 & 10 & 1.38 & 1.30 & 1.26 & 1.25 & 1.19 \\
30 & 9 & 50 & 1.48 & 1.30 & 1.28 & 1.28 & 1.28 \\
30 & 9 & 100 & 1.52 & 1.30 & 1.28 & 1.28 & 1.30
\end{tabular}
\caption{Average a-posteriori bounds.}\label{aposteriori}
\end{table}

Finally, we show more details on the a-posteriori bounds by providing both the upper and lower bounds in Tables~\ref{ubs} and \ref{lbs}. We find that our approach gives both better upper, and better lower bounds than the midpoint approach. While the MaxMin approach provides the best lower bounds, its upper bounds are often worse than for the midpoint solution.

\begin{table}
\centering
\begin{tabular}{rrr|r|r|rrr|r}
$n$ & $p$ & $N$ & OPT & Mid-UB & LP-1-UB & LP-2-UB & LP-3-UB & MM-UB \\
\hline
10 & 3 & 2 & 96.6 & 108.0 & 105.3 & 103.8 & 103.3 & 110.3 \\
10 & 3 & 5 & 142.9 & 169.5 & 162.8 & 158.0 & 158.8 & 165.9 \\
10 & 3 & 10 & 170.4 & 199.3 & 198.2 & 189.0 & 189.1 & 202.0 \\
10 & 3 & 50 & 219.0 & 248.3 & 249.8 & 241.9 & 239.9 & 254.1 \\
10 & 3 & 100 & 234.8 & 260.4 & 262.6 & 256.3 & 253.6 & 265.4 \\
20 & 6 & 2 & 172.1 & 193.7 & 190.6 & 188.9 & 187.5 & 189.1 \\
20 & 6 & 5 & 247.6 & 296.6 & 289.4 & 282.2 & 280.2 & 276.9 \\
20 & 6 & 10 & 292.7 & 351.0 & 346.2 & 334.6 & 332.3 & 337.2 \\
20 & 6 & 50 & 369.4 & 431.8 & 438.8 & 424.6 & 420.6 & 440.3 \\
20 & 6 & 100 & 395.6 & 457.7 & 461.2 & 450.8 & 446.6 & 464.5 \\
30 & 9 & 2 & 247.2 & 276.1 & 273.9 & 273.0 & 271.9 & 266.0 \\
30 & 9 & 5 & 351.2 & 416.2 & 408.6 & 398.3 & 395.9 & 384.1 \\
30 & 9 & 10 & 409.2 & 491.1 & 483.3 & 471.7 & 467.7 & 461.6 \\
30 & 9 & 50 & 513.1 & 605.5 & 610.3 & 592.4 & 588.6 & 607.0 \\
30 & 9 & 100 & 547.5 & 638.3 & 645.1 & 628.5 & 623.6 & 648.3
\end{tabular}
\caption{Average upper bounds.}\label{ubs}
\end{table}

\begin{table}
\centering
\begin{tabular}{rrr|r|r|rrr|r}
$n$ & $p$ & $N$ & OPT & Mid-LB & LP-1-LB & LP-2-LB & LP-3-LB & MM-LB \\
\hline
10 & 3 & 2 & 96.6 & 82.9 & 85.1 & 85.8 & 86.1 & 90.1 \\
10 & 3 & 5 & 142.9 & 108.3 & 121.9 & 122.1 & 121.1 & 129.4 \\
10 & 3 & 10 & 170.4 & 120.3 & 143.6 & 141.6 & 139.6 & 151.2 \\
10 & 3 & 50 & 219.0 & 136.7 & 183.0 & 178.2 & 174.4 & 186.2 \\
10 & 3 & 100 & 234.8 & 140.8 & 194.8 & 190.6 & 186.2 & 196.6 \\
20 & 6 & 2 & 172.1 & 160.5 & 161.6 & 162.1 & 162.4 & 166.2 \\
20 & 6 & 5 & 247.6 & 212.9 & 223.7 & 225.2 & 225.6 & 234.1 \\
20 & 6 & 10 & 292.7 & 238.5 & 260.9 & 261.2 & 260.3 & 272.9 \\
20 & 6 & 50 & 369.4 & 272.5 & 327.8 & 323.6 & 319.9 & 333.1 \\
20 & 6 & 100 & 395.6 & 280.8 & 348.1 & 343.8 & 339.7 & 351.2 \\
30 & 9 & 2 & 247.2 & 236.3 & 237.2 & 237.5 & 237.8 & 242.1 \\
30 & 9 & 5 & 351.2 & 316.3 & 325.0 & 328.3 & 328.9 & 337.9 \\
30 & 9 & 10 & 409.2 & 355.1 & 373.2 & 375.6 & 375.8 & 389.2 \\
30 & 9 & 50 & 513.1 & 408.0 & 467.8 & 464.3 & 460.7 & 475.3 \\
30 & 9 & 100 & 547.5 & 420.4 & 495.9 & 491.5 & 487.4 & 500.1
\end{tabular}
\caption{Average lower bounds.}\label{lbs}
\end{table}

\section{Conclusion}
Most robust combinatorial optimization problems are hard, which has lead to the development of general approximation algorithms. The two best-known such approaches are the midpoint method and the element-wise worst-case approach. Both rely on creating a single scenario that is representative for the whole uncertainty set. By reconsidering the respective proofs that both are $N$-approximation algorithms, we find an optimization problem to construct a representative scenario that results in an approximation which is at least as good as for the previous two scenarios.

In computational experiments using the selection problem, we test this approach numerically. We find that the midpoint method gives a guarantee that is about 20\% larger than ours, while we only need to solve a simple linear program to construct the representative scenario. The improved a-priori guarantee is also reflected in an improved a-posteriori guarantee, with our approach providing both better upper and lower bounds than before. This smaller gap could potentially be used within branch-and-bound algorithms for a more efficient search for an optimal solution.

\end{document}